\documentclass{amsart}

\usepackage{amssymb}
\usepackage[hidelinks]{hyperref}
\usepackage[capitalize]{cleveref}
\usepackage{microtype}

\newcommand{\N}{\mathbf{N}}
\newcommand{\F}{\mathbf{F}}
\newcommand{\E}{\mathbf{E}}
\newcommand{\C}{\mathbf{C}}

\newcommand{\abs}[1]{\lvert #1 \rvert}

\newtheorem{thm}{Theorem}
\newtheorem{cor}[thm]{Corollary}
\newtheorem{lem}[thm]{Lemma}

\theoremstyle{definition}

\newtheorem*{defn}{Definition}

\newcommand{\Span}{\operatorname{Span}}

\title{Simplices over finite fields}
\author{Hans Parshall}
\subjclass[2010]{11T24, 05D10}

\begin{document}

\begin{abstract}

We prove that, provided $d > k$, every sufficiently large subset of $\F_q^d$ contains an isometric copy of every $k$-simplex that avoids spanning a nontrivial self-orthogonal subspace.  We obtain comparable results for simplices exhibiting self-orthogonal behavior.

\end{abstract}

\maketitle

\section{Introduction}

Density results in geometric Ramsey theory are concerned with finding geometric configurations that must appear in any sufficiently large subset of a vector space.  A striking example of this is the theorem of Katznelson and Weiss~\cite{katznelsonweiss}, proved using ergodic theory, which states that every set of positive upper density within $\mathbf{R}^2$ contains two points distance $\lambda$ apart for every sufficiently large $\lambda>0$.  Bourgain~\cite{bourgain} gave an alternate, Fourier-analytic proof yielding more: provided $d > k$, every set of positive upper density within $\mathbf{R}^d$ contains the vertex set of every sufficiently large dilate of any non-degenerate Euclidean $k$-simplex.  An analogous result requiring $d > 2k + 4$ for subsets of the integer lattice $\mathbf{Z}^d$ was proved by Magyar~\cite{magyar} using Fourier analysis and a variant of the circle method.  

The main result of this paper improves the comparable situation over a finite field $\mathbf{F}_q$ with odd characteristic. In particular, we show that when $d > 2k$, every sufficiently large subset of $\mathbf{F}_q^d$ contains an isometric copy of every $k$-simplex, where here by $k$-simplex we mean a set of $k + 1$ points $\{v_0, v_1, \ldots, v_k\} \subseteq \mathbf{F}_q^d$ for which the vectors $\{v_j - v_0\}_{j = 1}^k$ are linearly independent.  Our proof is inspired by the simplified Fourier analytic argument for Bourgain's theorem presented by Lyall and Magyar \cite{lyall-magyar}, which relies on the well-known decay of the Fourier transform for the surface measure of the Euclidean unit sphere.  In our setting of $\mathbf{F}_q^d$, we will define weighted indicator functions associated to a reference simplex which should be thought of as spherical measures, and we find it necessary to carefully compute asymptotics for their Fourier transforms.  Kloosterman sums play a role in the estimation of our error terms.

The situation in $\F_q^d$ is somewhat delicate, since we will be working with an isotropic measurement of length.  To be clear, with $v,w \in \F_q^d$, we are always working with the usual dot product, defined by $v \cdot w := \sum_{j = 1}^d v_jw_j$, and we call $|v|^2 := v \cdot v$ the length of $v$.  This notion of length was studied, for example, by Iosevich and Rudnev \cite{iosevich-rudnev}, who study a finite field analogue of the Falconer distance problem and prove a satisfying finite field analogue to the aforementioned theorem of Katznelson and Weiss.  Other authors have worked with this length to prove variants of Bourgain's theorem in $\mathbf{F}_q^d$; see especially \cite{cehik}, \cite{hart-iosevich}, and \cite{vinh}.

We say two $k$-simplices $\Delta_k, \Delta_k' \subset \F_q^d$ are isometric, and write $\Delta_k \simeq \Delta_k'$ if they can be ordered 
\begin{align*}
\Delta_k &= \{v_0, v_1, \ldots, v_k\}\\
\Delta_k' &= \{y_0, y_1, \ldots, y_k\}
\end{align*}
so that dot products are preserved, in the sense that
\begin{equation}\label{dotProducts}
	(v_i - v_0) \cdot (v_j - v_0) = (y_i - y_0) \cdot (y_j - y_0)
\end{equation}
for all $1 \leq i \leq j \leq k$.  To the $k$-simplex $\Delta_k = \{v_0, v_1, \ldots, v_k\} \subset \F_q^d$, we associate the subspace 
\[
	V = \Span(v_1 - v_0, \ldots, v_k - v_0)
\]
and define its orthogonal complement as usual by 
\[
	V^\perp = \{w \in \F_q^d : v \cdot w = 0 \text{ for all } v \in V\}.
\]
We define the rank of $\Delta_k$ to be the quantity $k - \dim(V \cap V^\perp)$; it is easy to check that isometric $k$-simplices share the same rank.  When $\Delta_k$ has rank $k$, the only self-orthogonal element of $V$ is the zero vector, and we call $\Delta_k$ full rank.  It seems that self-orthogonality is the main obstruction to simplices behaving as one would expect when comparing to the Euclidean case, where all simplices are full rank.

Results in the direction we are heading have mostly restricted their attention to $k$-simplices $\Delta_k = \{v_0, v_1, \ldots, v_k\}$ for which $|v_j - v_i|^2 \neq 0$ for all $0 \leq i < j \leq k$, a condition that is certainly implied by $\Delta_k$ having full rank.  It was shown by Hart and Iosevich \cite{hart-iosevich} that when $d > \binom{k + 1}{2}$ and $C > 0$ is taken sufficiently large with respect to $k$, then every $A \subseteq \F_q^d$ with $|A| = \alpha q^d$ contains an isometric copy of every full rank $k$-simplex provided $\alpha \geq C q^{k/2 - d/(k + 1)}$.  This was subsequently improved by Vinh \cite{vinh}, where he was able to allow for $d \geq 2k$ and $\alpha \geq Cq^{k - (d + 1)/2}$.  Our main result here allows us to locate $k$-simplices of any rank, with the relationship between $d$ and $k$ improving for higher ranks.

\begin{thm}\label{mainResult}
	Let $k \geq 1$, $0 \leq r \leq k$, and $d > 2k - r$. Let $\Delta_k \subset \F_q^d$ be a $k$-simplex with rank $r$, and let $\alpha \geq Cq^{(2k - d - r)/(k + 1)}$ for some sufficiently large constant $C > 0$ depending only on $k$.  Then for any set $A \subseteq \F_q^d$ with $|A| = \alpha q^d$, the number of $k$-simplices $\Delta_k' \subset A$ with $\Delta_k' \simeq \Delta_k$ is equal to
	\[
		\Big(\alpha^{k + 1} + O(\alpha^{(k + 1)/2}q^{k - (d + r)/2})\Big) q^{(k + 1)d - \binom{k + 1}{2}}.
	\]
\end{thm}

As an immediate corollary, we achieve the following for full rank simplices.

\begin{cor}\label{nondegenerateImprovement}
	Let $d > k \geq 1$ and let $\alpha \geq Cq^{(k - d)/(k + 1)}$ for some sufficiently large constant $C > 0$ depending only on $k$.  If $A \subseteq \F_q^d$ with $|A| = \alpha q^d$, then $A$ contains an isometric copy of every full rank $k$-simplex.
\end{cor}

\cref{mainResult} is optimal in the sense that one would expect the number of $k$-simplices isometric to a fixed $\Delta_k$ within a uniformly random set $A \subseteq \F_q^d$ to be equal to $|A|^{(k + 1)} q^{-\binom{k + 1}{2}}$.  The condition $d > 2k - r$ does not leave much room for potential improvement.  That is, suppose $\Delta_k = \{v_0, v_1, \ldots, v_k\} \subset \F_q^d$ is a $k$-simplex of rank $r$.  Then setting $V = \Span(v_1 - v_0, \ldots, v_k - v_0)$, we have $\dim(V \cap V^\perp) = k - r$.  Certainly $V \cap V^\perp \subseteq V^\perp$, forcing the dimension relationship $k - r \leq d - k$ which establishes $d \geq 2k - r$ in every case.  Without restrictions on our parameters, this bound is sharp, since we need to account for the possibility of a subspace $W \subset \F_q^{2k - 2r}$ with $\dim(W) = k - r$ and $W = W^\perp$.  Such a subspace exists, for example, when $-1$ is a square modulo $q$.  One can then construct a $k$-simplex of rank $r$ in $\F_q^{2k - r} = \F_q^{r} \times \F_q^{2k - 2r}$ using the $r$ standard basis vectors for $\F_q^r \times \{0\}$ and the $k - r$ basis vectors for $\{0\} \times W$.  Hence, we should be attempting to handle $k$-simplices of rank $r$ whenever $d \geq 2k - r$.

This paper is organized as follows.  In Section 2, we provide some notation and background results that will be required as the paper unfolds.  In Section 3, we compute asymptotics for some spherical measures that will allow us to count $k$-simplices of a given isometry class in the final section.

\section{Setup}

Here we list some notation and facts that we will use as we progress. For any function $f : \F_q^d \rightarrow \C$, we will use the notation
\[
	\E_x f(x) := q^{-d} \sum_{x \in \F_q^d} f(x)
\]
for the average of $f$ over $\F_q^d$.  Let $\chi$ denote the canonical additive character of $\F_q$.  Then for $y \in \F_q^d$, we have the usual orthogonality relationship of
\[
	\E_x \chi(y \cdot x) = \begin{cases} 1 \text{ if } y = 0 \\ 0 \text{ otherwise} \end{cases}
\]
Letting $f,g : \F_q^d \rightarrow \C$, we define the Fourier transform $\widehat{f} : \F_q^d \rightarrow \C$ by
\[
	\widehat{f}(\xi) := \E_x f(x) \chi(-\xi \cdot x)
\]
and we recall the Fourier inversion formula
\[
	f(x) = \sum_{\xi \in \F_q^d} \widehat{f}(\xi) \chi(\xi \cdot x)
\]
and Plancherel's identity
\[
	\E_x f(x)\overline{g(x)} = \sum_{\xi \in \F_q^d} \widehat{f}(\xi)\overline{\widehat{g}(\xi)}.
\]
Defining the convolution $f * g(x) = \E_y f(y)g(x - y)$, we recall $\widehat{f * g}(\xi) = \widehat{f}(\xi)\widehat{g}(\xi)$.

Let $\eta$ denote the quadratic multiplicative character of $\F_q$.  We will require the following character sum estimate.

\begin{lem}\label{kloostermanEstimate}
	For $n \in \N$, $a \in \F_q^*$, and $b \in \F_q$,
	\[
	\Big| \sum_{s \in \F_q^*} \eta(s)^n \chi(as + b/s) \Big| = O(q^{1/2}).
	\]
\end{lem}
\begin{proof}
If $n$ is even and $b = 0$, orthogonality actually provides $\sum_{s \in \F_q^*} \chi(as) = -1$, but this savings will not be important to us. If $n$ is even and $b \neq 0$, this is a standard Kloosterman sum estimate; see Theorem 5.45 of \cite{finitefields}.  When $n$ is odd, we are either considering a Gauss sum when $b = 0$ or a Sali\'{e} sum when $b \neq 0$, and one could consult Theorem 5.15 of \cite{finitefields} or the proof of Lemma 12.4 of \cite{iwaniec-kowalski}, respectively.
\end{proof}
We will also repeatedly require the character sum identity
\begin{equation}\label{gaussSumIdentity}
	\sum_{x \in \F_q^d} \chi(a|x|^2 + b \cdot x) = G_q^d \eta(a)^d \chi(-|b|^2/4a),
\end{equation}
where $G_q$ is a complex number depending only on $q$ that satisfies $|G_q| = \sqrt{q}$.  This identity follows, for example, from Theorem 5.33 of \cite{finitefields}.

For the remainder, we fix $d > 2k - r$, a reference $k$-simplex of rank $r$
\[
	\Delta_k := \{0, v_1, \ldots, v_k\} \subset \F_q^d,
\]
and write $\Delta_j = \{0, v_1, \ldots, v_j\}$ for $1 \leq j \leq k - 1$.  Note there is no harm in taking $v_0 = 0$, since the condition \eqref{dotProducts} is translation invariant.  By reordering the vectors if necessary, we can insist that $\Delta_j$ has rank $r_j$ where
\[
	r_j := \begin{cases} j & \text{if } 0 \leq j \leq r \\ r & \text{if } r \leq j \leq k \end{cases}
\]
Whenever we write that two ordered $k$-simplices are isometric, we mean that \eqref{dotProducts} is satisfied with the given orderings. Finally, we agree that constants implied by big-O notation will often depend on the dimension of the reference simplex, $k$, but never on any of our other parameters.  

\section{Measures Associated to the Reference Simplex}

In this section, we study the following weighted measures associated to $\Delta_k$.

\begin{defn}
	For $y_1, \ldots, y_k \in \F_q^d$, define
	\begin{align*}
		\sigma(y_1) &:= \begin{cases} q & \text{if } |y_1|^2 = |v_1|^2 \\ 0 & \text{otherwise} \end{cases}\\
		\sigma_{y_1, \ldots, y_{j - 1}}(y_j)&:=\begin{cases} q^j & \text{if } y_i \cdot y_j = v_i \cdot v_j \text{ for all } 1 \leq i \leq j \\
		0 & \text{otherwise} \end{cases}
	\end{align*}
\end{defn}
	These can be used to detect whether a $k$-simplex $\Delta_k' = \{0, y_1, \ldots, y_k\} \subset \F_q^d$ is isometric to $\Delta_k$, in the sense that
\[
	\sigma(y_1)\sigma_{y_1}(y_2) \cdots \sigma_{y_1, \ldots, y_{k - 1}}(y_k) = \begin{cases} q^{\binom{k + 1}{2}} & \text{if } \Delta_k' \simeq \Delta_k \\
	0 & \text{otherwise}
	\end{cases}
\]
Our choice of weights may look strange initially, but we find it convenient for these measures to be (essentially) $L^1$-normalized.  Roughly, we will be using the measures $\sigma_{y_1, \ldots, y_{j - 1}}$ to inductively count $j$-simplices isometric to $\Delta_j$ by counting how many points ``complete'' each $(j - 1)$-simplex isometric to $\Delta_{j - 1}$.  This will eventually reduce to counting how many vectors $y_1 \in \F_q^d$ have a fixed length $|v_1|^2 \in \F_q$, for which we use $\sigma$.  For $\xi \in \F_q^d$, we define
\[
	\delta(\xi) := \begin{cases} 1 & \text{if } \xi = 0 \\ 0 & \text{otherwise} \end{cases}
\]
and record the standard estimate for $\widehat{\sigma}$.
\begin{lem}\label{sphereAsymptotic}
	If $|v_1|^2 \neq 0$, then
	\[
		\widehat{\sigma}(\xi) = \delta(\xi) + O(q^{(1 - d)/2}).
	\]
\end{lem}

\begin{proof}
	See, for example, Lemma 3.3 in \cite{hart-iosevich}.
\end{proof}

The proof for \cref{sphereAsymptotic} serves as a guide for our first asymptotic computation, for which we make the definition
\[
	\delta_{y_1, \ldots, y_{j - 1}}(\xi) := \begin{cases} 1 & \text{if } \xi \in \Span(y_1, \ldots, y_{j - 1}) \\ 0 & \text{otherwise} \end{cases}
\]

\begin{lem}\label{transformAsymptotic}
	If $2 \leq j \leq k$, $r_j = j$, and the ordered $(j - 1)$-simplex $\{0, y_1, \ldots, y_{j - 1}\}$ is isometric to $\Delta_{j - 1}$, then
	\[
		|\widehat{\sigma}_{y_1, \ldots, y_{j - 1}}(\xi)| = \delta_{y_1, \ldots, y_{j - 1}}(\xi) + O(q^{(j - d)/2}).
	\]
\end{lem}

\begin{proof}
	We consider
	\[
		\widehat{\sigma}_{y_1, \ldots, y_{j - 1}}(\xi) = \E_{x \in \F_q^d} \sigma_{y_1, \ldots, y_{j - 1}}(x) \chi(-\xi \cdot x).
	\]
	As $\Delta_{j - 1}$ is full rank and $\{0, y_1, \ldots, y_{j - 1}\} \simeq \Delta_{j - 1}$, Witt's extension theorem (see, for example, \cite{lam}) implies the existence of an isometry $U : \F_q^d \rightarrow \F_q^d$ such that $U(v_i) = y_i$; here by isometry we mean that $U$ is an orthogonal linear transformation for which $U(v) \cdot U(w) = v \cdot w$ for all $v,w \in \mathbf{F}_q^d$.	We set $z := U(v_j)$ and individually expand each of the $j$ conditions that $\sigma_{y_1, \ldots, y_{j - 1}}$ checks via orthogonality to see $\widehat{\sigma}_{y_1, \ldots, y_{j - 1}}(\xi)$ is equal to
\[
	\sum_{s \in \F_q} \chi(-s |z|^2) \sum_{t_1, \ldots, t_{j - 1} \in \F_q} \prod_{i = 1}^{j - 1} \chi(-t_iy_i \cdot z) \E_{x \in \F_q^d} \chi(s|x|^2 - \xi \cdot x)\prod_{i = 1}^{j - 1}\chi(t_iy_i \cdot x),
\]
where here we use that $y_i \cdot z = v_i \cdot v_j$ and $|z|^2 = |v_j|^2$.  Rather than summing over the $t_i$ coefficients, we can instead sum over the span of the $y_i$ vectors for
\[
	\widehat{\sigma}_{y_1, \ldots, y_{j - 1}}(\xi) = \sum_{s \in \F_q} \chi(-s|z|^2) \sum_{y \in \Span(y_1, \ldots, y_{j - 1})} \chi(-y \cdot z) \E_{x \in \F_q^d} \chi(s|x|^2 + (y - \xi) \cdot x).
\]
When $s = 0$, orthogonality guarantees that the average in $x$ becomes 1 exactly when $\xi = y$ but zeroes out otherwise.  Then we have
\[
\widehat{\sigma}_{y_1, \ldots, y_{j - 1}}(\xi) = \chi(-\xi \cdot z) \delta_{y_1, \ldots, y_{j - 1}}(\xi) + \mathcal{E},
\]
where we define
\[
	\mathcal{E} := \sum_{s \in \F_q^*} \chi(-s|z|^2) \sum_{y \in \Span(y_1, \ldots, y_{j - 1})} \chi(-y \cdot z) \E_{x \in \F_q^d} \chi(s|x|^2 + (y - \xi) \cdot x).
\]
It remains to show that we have the claimed error bound of $\mathcal{E} = O(q^{(j - d)/2})$.  From the identity \eqref{gaussSumIdentity},
\[
	|\mathcal{E}| = q^{-d/2} \Big|\sum_{s \in \F_q^*} \eta(s)^d \chi(-s |z|^2) \sum_{y \in \Span(y_1, \ldots, y_{j - 1})} \chi(-|y - \xi|^2/4s - z \cdot y)\Big|.
\]
Expanding $|y - \xi|^2 = |\xi|^2 - 2\xi \cdot y + |y|^2$ and applying the change of variables $y \mapsto 2sy$, we have
\[
	|\mathcal{E}| = q^{-d/2}\Big|\sum_{s \in \F_q^*} \eta(s)^d \chi(-s |z|^2 - |\xi|^2/4s) \sum_{y \in \Span(y_1, \ldots, y_{j - 1})} \chi(-s|y|^2 - (2sz - \xi) \cdot y)\Big|.
\]
We would like to again invoke \eqref{gaussSumIdentity}, so we let $\{u_1, \ldots, u_{j - 1}\}$ be an orthogonal basis for $\Span(y_1, \ldots, y_{j - 1})$.  Expressing the above in terms of this basis, applying \eqref{gaussSumIdentity} again, and rearranging, we have that $|\mathcal{E}|$ is equal to
\begin{align*}
	 &q^{-d/2} \Big|\sum_{s \in \F_q^*} \eta(s)^d \chi\Big(-s |z|^2 - \frac{|\xi|^2}{4s}\Big) \prod_{i = 1}^{j - 1} \sum_{a_i \in \F_q} \chi(-s|u_i|^2 a_i^2 - (2sz \cdot u_i - \xi \cdot u_i) a_i)\Big|\\
	&= q^{(j - d - 1)/2} \Big|\sum_{s \in \F_q^*} \eta(s)^{d + j - 1} \chi(-s |z|^2 - \frac{|\xi|^2}{4s}) \prod_{i = 1}^{j - 1} \chi\Big(\frac{s(z \cdot u_i)^2}{|u_i|^2} + \frac{(\xi \cdot u_i)^2}{4s|u_i|^2}\Big)\Big|.
\end{align*}
Setting
\begin{align*}
	a :=& \sum_{i = 1}^{j - 1} \frac{(z \cdot u_i)^2}{|u_i|^2} - |z|^2\\
	b :=& \sum_{i = 1}^{j - 1} (\xi \cdot u_i)^2/4|u_i|^2 - |\xi|^2/4,
\end{align*}
we have shown
\[
	|\mathcal{E}| = q^{(j - d - 1)/2} \Big|\sum_{s \neq 0} \eta(s)^{d + j - 1} \chi(as + b/s)\Big|
\]
and our claimed error bound of $|\mathcal{E}| = q^{(j - d)/2}$ follows from \cref{kloostermanEstimate} provided $a \neq 0$.  If $a = 0$, then the vector
\[
	\sum_{i = 1}^{j - 1} \Big(\frac{z \cdot u_i}{|u_i|^2}\Big) u_i - z
\]
would be a self-orthogonal member of $\Span(u_1, \ldots, u_{j - 1},z)$.  Of course, this subspace is the orthogonal image of $\Span(v_1, \ldots, v_j)$, in which case our assumption that $\Delta_j$ has rank $j$ guarantees $a \neq 0$.
\end{proof}

For degenerate simplices, we cannot hope to obtain as much cancellation in our character sums.  The simplest case is to compute an asymptotic for the Fourier transform of the sphere of radius 0.

\begin{lem}\label{zeroSphereAsymptotic}
	If $|v_1|^2 = 0$, then
	\[
		\widehat{\sigma}(\xi) = \delta(\xi) + O(q^{1 - d/2}).
	\]
\end{lem}

\begin{proof}
	By orthogonality, we are considering
	\[
		\widehat{\sigma}(\xi) = \sum_{s \in \F_q} \E_x \chi(s|x|^2 - \xi \cdot x).
	\]
	Separating the $s = 0$ term and applying \eqref{gaussSumIdentity},
	\[
		\widehat{\sigma}(\xi) = \delta(\xi) + q^{-d} G_q^{d/2} \sum_{s \in \F_q^*} \eta(t)^d \omega^{-|\xi|^2/4s}
	\]
	While we may be able to hope for some cancellation in the case $d$ is odd or $|\xi|^2 \neq 0$, these specific savings do not improve any of our applications, so we content ourselves to bounding the error trivially by $O(q^{1 - d/2})$.
\end{proof}

In light of the lack of cancellation in the self-orthogonal case, we avoid the identity \eqref{gaussSumIdentity} and the Kloosterman sum machinery in favor of Weyl differencing.  This simpler method still provides what we believe to be the right order of magnitude for our error terms in the low rank case.

\begin{lem}\label{degenerateAsymptotic}
	If $2 \leq j \leq k$, $r_j < j$, and the ordered $(j - 1)$-simplex $\{0, y_1, \ldots, y_{j - 1}\}$ is isometric to $\Delta_{j - 1}$, then
	\[
		\abs{\widehat{\sigma}_{y_1, \ldots, y_{j - 1}}(\xi)} = \delta_{y_1, \ldots, y_{j - 1}}(\xi) + O(q^{j - (d + r_j)/2}).
	\]
\end{lem}

\begin{proof}
	As before, we consider
	\[
		\widehat{\sigma}_{y_1, \ldots, y_{j - 1}}(\xi) = \E_x \sigma_{y_1, \ldots, y_{j - 1}}(x) \chi(-\xi \cdot x).
	\]
	Let $z \in \F_q^d$ such that $y_i \cdot z = v_i \cdot v_k$ for $1 \leq i \leq j - 1$.  Setting $V = \Span(y_1, \ldots, y_{j - 1})$, the condition $\sigma_{y_1, \ldots, y_{j - 1}}(x) \neq 0$ restricts our attention to $x \in z + V^\perp$ with $|x|^2 = |v_j|^2$.  Reindexing our sum accordingly and expanding the condition $|x|^2 = |v_j|^2$ via orthogonality, we have
	\[
		\widehat{\sigma}_{y_1, \ldots, y_{j - 1}}(\xi) = q^{j - d - 1} \chi(-\xi \cdot z)\sum_{s \in \F_q} \chi(-s|v_j|^2) \sum_{x \in V^\perp} \chi(s|x + z|^2 - \xi \cdot x).
	\]
	When $s = 0$, we are left with the inner sum of
	\[
		\sum_{x \in V^\perp} \chi(-\xi \cdot x) = \begin{cases} q^{d - j + 1} & \text{if } \xi \in V \\ 0 & \text{otherwise} \end{cases}
	\]
	From this, we have
	\[
		\widehat{\sigma}_{y_1, \ldots, y_{j - 1}}(\xi) = \chi(-\xi \cdot z) \delta_{y_1, \ldots, y_{j - 1}}(\xi) + \mathcal{E},
	\]
	where we define
	\begin{align*}
		\mathcal{E} &:= q^{j - d - 1} \chi(-\xi \cdot z)\sum_{s \in \F_q^*} \chi(-s|v_j|^2) \sum_{x \in V^\perp} \chi(s|x + z|^2 - \xi \cdot x),\\
		&= q^{j - d - 1} \chi(-\xi \cdot z) \sum_{s \in \F_q^*} \chi(s|z|^2 - s|v_j|^2) \sum_{x \in V^\perp} \chi(s|x|^2 + (2sz - \xi) \cdot x).
	\end{align*}
	From the triangle inequality, we certainly have
	\[
		|\mathcal{E}| \leq q^{j - d} \max_{s \in \F_q^*} \Big|\sum_{x \in V^\perp} \chi(s|x|^2 + (2sz - \xi) \cdot x)\Big|.
	\]
	By expanding and rearranging,
	\begin{align*}
		\Big|\sum_{x \in V^\perp} \chi(s|x|^2 + (2sz - \xi) \cdot x)\Big|^2 &= \sum_{x,h \in V^\perp} \chi(s|x + h|^2 - s|x|^2 + (2sz - \xi) \cdot h)\\
		&= \sum_{h \in V^\perp} \chi(s|h|^2 + (2sz - \xi) \cdot h) \sum_{x \in V^\perp} \chi(2sh \cdot x)\\
		&\leq \sum_{h \in V^\perp} \Big|\sum_{x \in V^\perp} \chi(2sh \cdot x)\Big|.
	\end{align*}
	Since we are working with $s \neq 0$, we are left with an inner sum of
	\[
		\sum_{x \in V^\perp} \chi(2sh \cdot x) = \begin{cases} q^{d - j + 1} & \text{if } h \in V \\ 0 & \text{otherwise} \end{cases}
	\]
	establishing $|\mathcal{E}|^2 \leq q^{j - d + 1} |V \cap V^\perp|$.  As the rank of $\Delta_{j - 1}$ is $r$, we have 
	\[
		\dim(V \cap V^\perp) = j - 1 - r
	\]
	implying $|\mathcal{E}|^2 \leq q^{2j - d - r}$, from which our claimed bound follows.
\end{proof}

To summarize, we can combine \cref{sphereAsymptotic} and \cref{zeroSphereAsymptotic} for
\begin{equation}\label{mainSphereAsymptotic}
	\widehat{\sigma}(\xi) = \delta(\xi) + O(q^{1 - (d + r_1)/2}),
\end{equation}
and by combining \cref{transformAsymptotic} and \cref{degenerateAsymptotic} we have, for every $2 \leq j \leq k$,
\begin{equation}\label{mainAsymptotic}
	\abs{\widehat{\sigma}_{y_1, \ldots, y_{j - 1}}(\xi)} = \delta_{y_1, \ldots, y_{j - 1}}(\xi) + O(q^{j - (d + r_j)/2}).
\end{equation}

\section{Counting Simplices}

In order to count $k$-simplices in $\F_q^d$ which are isometric to $\Delta_k$, we introduce notation for $1 \leq j \leq k$ and $f_0, \ldots, f_j : \F_q^d \rightarrow [-1,1]$, 
\begin{align*}
	S_j(y_1, \ldots, y_j) &:= \sigma(y_1)\sigma_{y_1}(y_2) \cdots \sigma_{y_1, \ldots, y_{j - 1}}(y_j)\\
	\mathcal{S}_j(f_0, \ldots, f_j) &:= \E_{y_1, \ldots, y_j}^* S_j(y_1, \ldots, y_j) \E_x f_0(x) \prod_{i = 1}^j f_i(x + y_i)
\end{align*}
where here we use the shorthand
\[
	\E_{y_1, \ldots, y_j}^*F(y_1, \ldots, y_j) := q^{-jd} \sum_{\substack{y_1, \ldots, y_j \in \F_q^d \\ \text{linearly independent}}} F(y_1, \ldots, y_j).
\]
We begin by showing that, while restricting the dot products of $y_1, \ldots, y_j$ is not enough to guarantee linear independence, the contribution of linearly dependent vectors is negligible in the following sense.

\begin{lem}\label{dependentError}
	If $2 \leq j \leq k$ and the ordered $(j - 1)$-simplex $\{0, y_1, \ldots, y_{j - 1}\}$ is isometric to $\Delta_{j - 1}$, then
	\[
		\sum_{y_j \in \Span(y_1, \ldots, y_{j - 1})} \sigma_{y_1, \ldots, y_{j - 1}}(y_j) \leq q^{2j - 1 - r_{j - 1}}.
	\]
\end{lem}

\begin{proof}
	Set $V = \Span(y_1, \ldots, y_{j - 1})$.  By fixing the dot products $y_1 \cdot y_j, \ldots, y_{j - 1} \cdot y_j$, we are restricting our attention to $y_j$ within a coset of $V \cap V^\perp$.  Hence, there are at most $q^{j - 1 - r_{j - 1}}$ choices for $y_j \in V$ for which $\sigma_{y_1, \ldots, y_{j - 1}}(y_j) = q^j$.
\end{proof}

Letting $\mathbf{1} : \F_q^d \rightarrow \{1\}$, this allows us to prove the following count.  

\begin{lem}\label{simplicesAsymptotic} For each $1 \leq j \leq k$,
\[
	\mathcal{S}_j(\mathbf{1}, \ldots, \mathbf{1}) = 1 + O(q^{j - (d + r_j)/2}).
\]
\end{lem}
\begin{proof}
	We proceed by induction on $j$.  When $j = 1$, this follows immediately from \eqref{mainSphereAsymptotic}.  When $j = 2$, we apply \cref{dependentError} and \eqref{mainSphereAsymptotic} for
	\begin{align*}
		\mathcal{S}_2(\mathbf{1}, \mathbf{1}, \mathbf{1}) &= \E_{y_1, y_2}^* \sigma(y_1) \sigma_{y_1}(y_2)\\
		&= \E_{y_1} \sigma(y_1) \E_{y_2} \sigma_{y_1}(y_2) + O(q^{3 - d - r_1})
	\end{align*}
	in which case the conclusion follows from \eqref{mainSphereAsymptotic}, \eqref{mainAsymptotic}, and the observation that $r_1 \geq r_2 - 1$.  For $j \geq 3$, suppose the lemma has been established for $j - 1 \geq 2$.  Applying \cref{dependentError} and our induction hypothesis,
	\begin{align*}
		\mathcal{S}_j(\mathbf{1}, \ldots, \mathbf{1}) &= \E_{y_1, \ldots, y_{j - 1}}^* S_{j - 1}(y_1, \ldots, y_{j - 1}) \E_{y_j} \sigma_{y_1, \ldots, y_{j - 1}}(y_j) + O(q^{2j - 1 - r_{j - 1}})
	\end{align*}
	in which case we finish by applying \eqref{mainAsymptotic} and our induction hypothesis once more, where here we use that the error $O(q^{j - (d + r_j)/2})$ increases with $j$.
\end{proof}

Since we are working with $d > 2k  - r$, \cref{simplicesAsymptotic} demonstrates that $\mathbf{F}_q^d$ contains roughly $q^{kd - \binom{k + 1}{2}}$ isometric copies of $\Delta_k$.  In order to count isometric copies of $\Delta_k$ within a set, we will need the following observation.

\begin{lem}\label{errorlemma}
	If $2 \leq j \leq k$, the ordered $(j - 1)$-simplex $\{0, y_1, \ldots, y_{j - 1}\}$ is isometric to $\Delta_{j - 1}$, and $\xi \in \mathbf{F}_q^d \setminus \{0\}$, then
	\[
		\E_{y_1, \ldots, y_{j - 1}}^* S_{j - 1}(y_1, \ldots, y_{j - 1}) |\widehat{\sigma}_{y_1, \ldots, y_{j - 1}}(\xi)|^2 = O(q^{2j - d - r_{j}}).
	\]
\end{lem}

\begin{proof}
	From \eqref{mainAsymptotic},
	\[
		|\widehat{\sigma}_{y_1, \ldots, y_{j - 1}}(\xi)|^2 = (1 + O(q^{j - (d + r_j)/2}))\delta_{y_1, \ldots, y_{j - 1}}(\xi) + O(q^{2j - d - r_j}),
	\]
	so by applying \cref{simplicesAsymptotic} we see it is enough to establish
	\begin{equation}\label{instead}
		\E_{y_1, \ldots, y_{j - 1}}^* S_{j - 1}(y_1, \ldots, y_{j - 1}) \delta_{y_1, \ldots, y_{j - 1}}(\xi) = O(q^{2j - d - r_j})
	\end{equation}
	which we show by induction on $j \geq 2$.  When $j = 2$, this is obvious from the trivial
	\[
		\sum_{y_1 \in \Span(\xi)} \sigma(y_1) \leq q^2 \leq q^{4 - r_2}.
	\]
	Suppose \eqref{instead} has been established for $j - 1 \geq 2$ in place of $j$.  We apply the bound
	\[
		\delta_{y_1, \ldots, y_{j - 1}}(\xi) \leq \delta_{y_1, \ldots, y_{j - 2}}(\xi) + \delta_{y_1, \ldots, y_{j - 2},\xi}(y_{j - 1}),
	\]
	which breaks our estimate for \eqref{instead} into two cases.  For the first case, we can include $y_{j - 1} \in \Span(y_1, \ldots, y_{j - 2})$ with \cref{dependentError}, apply \eqref{mainAsymptotic}, and invoke our induction hypothesis for
	\[
		\E_{y_1, \ldots, y_{j - 1}}^* S_{j - 1}(y_1, \ldots, y_{j - 1}) \delta_{y_1, \ldots, y_{j - 2}}(\xi) = O(q^{2j - 2 - d - r_{j - 1}}),
	\]
	which is better than required since $r_{j - 1} \geq r_j - 1$.  For the second case, we need to count $y_{j - 1} \in \Span(y_1, \ldots, y_{j - 2},\xi)$ for which $\sigma_{y_1, \ldots, y_{j - 2}}(y_{j - 1}) = q^{j - 1}$.  We have $q$ choices for the $\xi$ coordinate of such $y_{j - 1}$ and, arguing as in \cref{dependentError}, at most 
	\[
		q^{j - 2 - r_{j - 2}} \leq q^{j - r_j}
	\]
	choices for the $y_1, \ldots, y_{j - 2}$ coordinates, since the dot product of $y_{j - 1}$ with each of $y_1, \ldots, y_{j - 2}$ is predetermined.  In total we conclude
	\[
		\sum_{y_{j - 1} \in \Span(y_1, \ldots, y_{j - 2}, \xi)} \sigma_{y_1, \ldots, y_{j - 2}}(y_{j - 1}) \leq q^{2j - r_j},
	\]
	in which case we can apply \cref{simplicesAsymptotic} for
	\[
		\E_{y_1, \ldots, y_{j - 1}}^* S_{j - 1}(y_1, \ldots, y_{j - 1}) \delta_{y_1, \ldots, y_{j - 2},\xi}(y_{j - 1}) = O(q^{2j - d - r_j}),
	\]
	establishing \eqref{instead}.
\end{proof}

To prove that large subsets contain $k$-simplices, we essentially show that, on average, any ordered $(j - 1)$-simplex isometric to $\Delta_{j - 1}$ contained in a set $A$ can be completed into a $j$-simplex isometric to $\Delta_j$ in roughly $|A|q^{-j}$ ways.  One should compare this approach with the simplified proof of Bourgain's result in the Euclidean setting presented in \cite{lyall-magyar}.  \cref{mainResult} follows immediately from the next theorem.

\begin{thm}\label{gvnLemma} Let $A \subseteq \F_q^d$ with $|A| = \alpha q^d$ where $\alpha \geq Cq^{(2k - d - r)/(k + 1)}$ for some sufficiently large constant $C > 0$ depending only upon $k$. Then for $1 \leq j \leq k$,
\begin{equation}\label{gvnCount}
	\mathcal{S}_j(1_A, \ldots, 1_A) = \alpha^{j + 1} + O(\alpha^{(j + 1)/2} q^{j - (d + r_j)/2}).
\end{equation}
\end{thm}

\begin{proof}
	We will proceed by induction on $j \geq 1$.  We start by setting $f_A = 1_A - \alpha$ and apply Plancherel for 
	\[
		\mathcal{S}_1(1_A, f_A) = \E_x 1_A(x) f_A * \sigma(x) = \sum_{\xi \in \mathbf{F}_q^d} |\widehat{f}_A(\xi)|^2 \widehat{\sigma}(\xi).
	\]	
	Since $\widehat{f}_A(0) = 0$, we can apply the triangle inequality, \eqref{mainSphereAsymptotic} and Plancherel for
	\[
		|\mathcal{S}_1(1_A, f_A)| = O(\alpha q^{1 - (d + r_1)/2}).
	\]
	The theorem then follows for $j = 1$ from the decomposition
	\begin{align*}
		\mathcal{S}_1(1_A, 1_A) &= \alpha \mathcal{S}_1(1_A, \mathbf{1}) + \mathcal{S}_1(1_A, f_A)\\
		&= \alpha \E_x 1_A(x) \E_y \sigma(y) + O(\alpha q^{1 - (d + r_1)/2})\\
		&= \alpha^2 + O(\alpha q^{1 - (d + r_1)/2}).
	\end{align*}
	In order to establish the theorem for $j \geq 2$, suppose \eqref{gvnCount} has been established with $j - 1$ in place of $j$.  We apply the decomposition
	\[
		\mathcal{S}_j(1_A, \ldots, 1_A) = \alpha \mathcal{S}_j(1_A, \ldots, 1_A, \mathbf{1}) + \mathcal{S}_j(1_A, \ldots, 1_A, f_A)
	\]
	and handle each term separately.  Arguing as in \cref{simplicesAsymptotic},
	\[
		\mathcal{S}_j(1_A, \ldots, 1_A,\mathbf{1}) = (1 + O(q^{j - (d + r_j)/2}) \mathcal{S}_{j - 1}(1_A, \ldots, 1_A),
	\]
	so by our induction hypothesis
	\[
		\alpha \mathcal{S}_j(1_A, \ldots, 1_A, \mathbf{1}) = \alpha^{j + 1} + O(\alpha^{j/2 + 1} q^{j - 1 - (d + r_{j - 1})/2}) + O(\alpha^{j + 1} q^{j - (d + r_j)/2}).
	\]
	As $r_{j - 1} \geq r_j - 1$, this is slightly stronger than the required
	\[
		\alpha \mathcal{S}_j(1_A, \ldots, 1_A, \mathbf{1}) = \alpha^{j + 1} + O(\alpha^{(j + 1)/2} q^{j - (d + r_j)/2}).
	\]
	We still need to show that $\mathcal{S}_j(1_A, \ldots, 1_A, f_A)$ contributes an acceptable error.  By rearranging and applying the triangle inequality, $|\mathcal{S}_j(1_A, \ldots, 1_A, f_A)|$ is at most
	\[
		\E_{y_1, \ldots, y_{j - 1}}^* S_{j - 1}(y_1, \ldots, y_{j - 1}) \E_x 1_A(x)1_A(x + y_1) \cdots 1_A(x + y_{j - 1}) |f_A * \tilde{\sigma}_{y_1, \ldots, y_{j - 1}}(x)|
	\]
	where here we let $\tilde{\sigma}_{y_1, \ldots, y_{j - 1}}(y_j) = \sigma_{y_1, \ldots, y_{j - 1}}(-y_j)$.
	Applying Cauchy-Schwarz, $|\mathcal{S}_j(1_A, \ldots, 1_A, f_A)|^2$ is at most
	\[
		\mathcal{S}_{j - 1}(1_A, \ldots, 1_A) \E_{y_1, \ldots, y_{j - 1}}^* S_{j - 1}(y_1, \ldots, y_{j - 1}) \E_x |f_A * \tilde{\sigma}_{y_1, \ldots, y_{j - 1}}|^2.
	\]
	By Plancherel, we have
	\[
		\E_x |f_A * \tilde{\sigma}_{y_1, \ldots, y_{j - 1}}|^2 = \sum_{\xi \in \mathbf{F}_q^d} |\widehat{f}_A(\xi)|^2 |\widehat{\sigma}_{y_1, \ldots, y_{j - 1}}(\xi)|^2,
	\]	
	in which case we can rearrange and apply \cref{errorlemma} and Plancherel for
	\[
		\E_{y_1, \ldots, y_{j - 1}}^* S_{j - 1}(y_1, \ldots, y_{j - 1}) \E_x |f_A * \tilde{\sigma}_{y_1, \ldots, y_{j - 1}}|^2 = O(\alpha q^{2j - d - r_j}).
	\]
	Invoking our induction hypothesis, we have shown
	\[
		|\mathcal{S}_j(1_A, \ldots, 1_A, f_A)|^2 = O(\alpha^{j + 1} q^{2j - d - r_j}) + O(\alpha^{j/2 + 1} q^{2j - d - r_j} q^{j - 1 - (d + r_{j - 1})/2}).
	\]
	By our hypothesis that $\alpha$ is not too small, the $O(\alpha^{j + 1} q^{2j - d - r_j})$ error dominates, completing the proof.
	\end{proof}
	
\textit{Acknowledgments.} The author thanks Neil Lyall and \'{A}kos Magyar for many helpful conversations.  This work also benefited from useful comments by Giorgis Petridis and the anonymous referee.

\bibliography{ff-simplices}{}
\bibliographystyle{amsplain}

\end{document}